\documentclass[leqno,oneside]{amsart}
\usepackage{amsfonts}
\usepackage{amsmath,amsthm,mathrsfs}
\usepackage{amsfonts,amssymb,geometry}
\usepackage{CJK}
\usepackage[all]{xy}
\usepackage{enumerate}

\newtheorem{thm}{Theorem}[section]
\newtheorem{cor}[thm]{Corollary}
\newtheorem{lem}[thm]{Lemma}
\newtheorem{prop}[thm]{Proposition}

\theoremstyle{remark}
\newtheorem{rem}[thm]{Remark}

\numberwithin{equation}{section}

\newcommand{\al}{\alpha}

\def \b{\beta}

\def\vz{\varepsilon}

\def\Lz{\Lambda}

\def\sz{\sigma}

\def\({\Bigl(}
\def \){ \Bigr)}

\def\x{{\bf x}}
\def\y{{\bf y}}

 \def\sz{{\sigma}}

 \def\RR{{\mathbb R}}

\def\sz{\sigma}

\def\x{{\bf x}}
\def\y{{\bf y}}
\def\a{{\mathbf a}}
 \def\b{{\mathbf b}}

\def\kk{{\bf k}}

\def\Lz{\Lambda}

\def\y{{\bf y}}

\def\Lz{\Lambda}

\begin{document}
\def\RR{\mathbb{R}}
\def\Exp{\text{Exp}}
\def\FF{\mathcal{F}_\al}
\title[] {Average Nikolskii factors for random trigonometric polynomials}

\author[]{Yun Ling} \address{School of Mathematical Sciences, Capital Normal
University, Beijing 10004,
China}
\email{18158616224@163.com}
\author[]{Jiaxin Geng} \address{School of Mathematical Sciences, Capital Normal
University, Beijing 100048,
China}
\email{gengjiaxin1208@163.com}
\author[]{Jiansong Li} \address{School of Mathematical Sciences, Capital Normal
University, Beijing 100048,
China}
\email{cnuljs2023@163.com}
\author[]{Heping Wang} \address{School of Mathematical Sciences, Capital Normal
University, Beijing 100048,
China}
\email{wanghp@cnu.edu.cn}

\begin{abstract}
For $1\le p,q\le \infty$, the Nikolskii factor for a trigonometric
polynomial $T_{\a}$  is defined by
$$\mathcal N_{p,q}(T_{\a})=\frac{\|T_{\a}\|_{q}}{\|T_{\a}\|_{p}},\ \ T_{\a}(x)=a_{1}+\sum\limits^{n}_{k=1}(a_{2k}\sqrt{2}\cos
kx+a_{2k+1}\sqrt{2}\sin kx).$$ We study this average (expected)
Nikolskii factor for random trigonometric  polynomials with
independent $N(0,\sigma^{2})$ coefficients and obtain that the
exact order. For $1\leq p<q<\infty$, the average Nikolskii factor
is  order degree to the 0, as compared to the degree $1/p-1/q$
worst case bound. We also give the generalization   to random
multivariate  trigonometric polynomials.
\end{abstract}
\maketitle
\input amssym.def

\section{Introduction}

Let $L_p(\Bbb T)\ (1\le p<\infty)$ be the usual Lebesgue space
consisting of measurable functions $f$ with finite norm
$$\|f\|_p=\Big(\frac1{2\pi}\int_{\Bbb T}|f(x)|^pdx\Big)^{1/p},$$
and $L_\infty(\Bbb T)$ be the space $C(\Bbb T)$ of continuous
functions on $\Bbb T$ with  norm $\|f\|_\infty=\max\limits_{x\in
\Bbb T}|f(x)|$, where $\mathbb{T}=[0,2\pi]$ is the torus.

Let  $\mathcal{T}_{n}$ be the  space of all trigonometric
polynomials $T_\a$ of degree $\le  n$, where
$\a=(a_{1},\ldots,a_{2n+1})\in\Bbb R^{2n+1}$,  $T_\a$ is of the
form
\begin{equation}\label{1.1}T_{\a}(x)=a_{1}+\sum\limits^{n}_{k=1}(a_{2k}\sqrt{2}\cos
kx+a_{2k+1}\sqrt{2}\sin kx).\end{equation}
The classical Nikolskii inequality \cite{T} states that there is a constant
$C>0$ such that for any trigonometric polynomial $T_{\a}$ of
degree $\le n$,
\begin{equation}\label{1.2}
\|T_{\a}\|_{q}\leq Cn^{(1/p-1/q)_+}\|T_{\a}\|_{p},
\end{equation} for any $1\leq p, q\leq\infty$, where $a_+=a$ if
$a\ge 0$ and $a_+=0$ if $a<0$. The order $(1/p-1/q)_+$ in the
above inequality \eqref{1.2} is sharp.

\par Alternatively we introduce, for given $p,q$, $1\le p,q\le \infty$,  the Nikolskii
factor for $T_\a$ is defined by
$$\mathcal{N}_{p,q}(T_{\a}):=\frac{\|T_{\a}\|_{q}}{\|T_{\a}\|_{p}}.$$
The worst Nikolskii factor for $\mathcal{T}_n$ is defined by
$$N^{\rm wor}_{p,q}(\mathcal{T}_{n}):=\sup_{0\neq T_{\a}\in\mathcal{T}_{n}}\mathcal{N}_{p,q}
(T_{\a})=\sup_{0\neq
T_{\a}\in\mathcal{T}_{n}}\frac{\|T_{\a}\|_{q}}{\|T_{\a}\|_{p}},$$
with the supremum being taken over all  trigonometric polynomial
of degree $\le n$ whose coefficients are not all zero. Then the
above Nikolskii inequality states
\begin{equation}\label{1.3-0}N^{\rm
wor}_{p,q}(\mathcal{T}_{n})\asymp n^{(1/p-1/q)_{+}},\end{equation}
i.e., that worst case bound for $\mathcal{N}_{p,q}(T_{\a})$ is
$\Theta(n^{(1/p-1/q)_{+}})$. Here we use the notion $A_n\asymp
B_n$ to express $A_n\ll B_n$ and $A_n\gg B_n$, and $A_n\ll B_n$
($A_n\gg B_n$) means that there exists a constant $c>0$
independent of $n$ such that $A_n\leq cB_n$ ($A_n\geq cB_n$).

In this paper, we consider the expected Nikolskii factor for a
random trigonometric polynomial. Given a  polynomial basis
$p_1,\dots, p_N$, a random polynomial $p$ is a polynomial whose
coefficients are random, i.e., $$p=\sum_{k=1}^NX_kp_k,$$where
$X_k,\,1\le k\le N$ are i.i.d random variables. In recent years
the study of random polynomials has attracted much interest, and
there are numerous papers devoted to this field. For example,  for
zeros of random polynomials, see \cite{BLR, CP, DNN, DNV, Pr, SW};
for norms, see \cite{BL, CE, DM, F, G,  NB, SZ}; for average
Markov factors, see \cite{Bos, LWW, PR, WYZ}.

There are also many  papers devoted to studying the Nikolskii type
inequalities, see for example, \cite{AD, D, DT, GT, T}. However,
as far as we know, there are no papers devoted to investigating
the average Nikolskii factors for random polynomials. In this
paper, we consider  random trigonometric polynomials $T_{\a}$
given by \eqref{1.1}, where $\a$ is a Gaussian random vector with
mean 0 and covariance  matrix $\sigma^2 I_N$, $\sz>0$, $I_N$ is
the $N$ by $N$ identity matrix, $N=\dim \mathcal T_n=2n+1$. For
$1\leq p,q\leq\infty$, we discuss the average Nikolskii factor
$$N^{\rm
ave}_{p,q}(\mathcal{T}_{n}):=\mathbb{E}\Big(\frac{\|T_{\a}\|_{q}}{\|T_{\a}\|_{p}}\Big).$$Clearly,
we have $$N^{\rm ave}_{p,q}(\mathcal{T}_{n})\le N^{\rm
wor}_{p,q}(\mathcal{T}_{n}).$$  It turns out that in many cases,
the average Nikolskii factor $N^{\rm ave}_{p,q}(\mathcal{T}_{n})$
 is significantly smaller than the worst Nikolskii factor
$N^{\rm wor}_{p,q}(\mathcal{T}_{n})$. We state our main result as
follows.
\begin{thm}\label{thm1.1}
Let  $1\leq p,q\leq\infty$, $N=2n+1$. Then we have
\begin{equation}\label{1.4-0}
N^{\rm ave}_{p,q}(\mathcal{T}_{n})=\mathbb{E}\Big(\frac{\|T_{\a}\|_{q}}{\|T_{\a}\|_{p}}\Big)\asymp\left\{
\begin{aligned}
 &1,&\ \,1\leq p,q<\infty\ or\ p=q=\infty,\\
 &(\ln N)^{1/2},&\ \,1\leq p<\infty,q=\infty,\\
 &(\ln N)^{-1/2},&\ \,1\leq q<\infty,p=\infty.
\end{aligned}
\right.
\end{equation}
\end{thm}

\begin{rem}
From \eqref{1.3-0} and \eqref{1.4-0} we know that the order of
the worst and average  Nikolskii factors for $1\leq q\leq
p<\infty$ or $p=q=\infty$ are same, which both are equivalent to
the constant $1$. Whereas, for other cases, the orders of average
Nikolskii factors are significantly smaller than the worst case
Nikolskii factors. More precisely, for $1\leq p<q<\infty$, the
worst case Nikolskii factor is order degree to the $1/p-1/q$,
while the average case Nikolskii factors is order degree to the 0;
for $1\leq p<q=\infty$, the worst case Nikolskii factor is order
degree to the $1/p$, while the order of the average Nikolskii
factor is  $(\ln N)^{1/2}$; for $1\le q<p=\infty$, the worst case
Nikolskii factor is order degree to 0, while the order of the
average Nikolskii factor is $(\ln N)^{-1/2}$.

This indicates that for $1\le p<q\le \infty$ or $1\le q<p=\infty$,
the average Nikolskii factor $N^{\rm ave}_{p,q}(\mathcal{T}_{n})$
 is significantly smaller than the worst Nikolskii factor
$N^{\rm wor}_{p,q}(\mathcal{T}_{n})$.
\end{rem}

\begin{rem}Let $\{\vz_i\}_{i=1}^{2n+1}$ be  a sequence  of
independent random variables on some probability space $(\Omega,
A, P)$ taking the values $\pm 1$ with probability $1/2$, that is
symmetric Bernoulli or Rademacher random variables. Let
$$T_{\vz}(x)=\vz_1 a_{1}+\sum\limits^{n}_{k=1}(\vz_{2k}\sqrt{2}\cos
kx+\vz_{2k+1}\sqrt{2}\sin kx)$$be a Rademacher random polynomial.
Then we have for $1\le p,q<\infty$,
$$\mathbb{E}\Big(\frac{\|T_{\vz}\|_{q}}{\|T_{\vz}\|_{p}}\Big)\asymp
1.$$

\end{rem}

We organize this paper as follows. In Section 2, we give the
estimate of average Nikolskii factor for general formulation.
Section 3 contains two subsections. In subsection 3.1, we
determine the exact order of the average Nikolskii factor in the
case of $1\leq q\leq\infty,\, p=2$. In subsection 3.2, we give the
upper bound of $\mathbb{E}(1/\|T_{\a}\|_{\infty})^r$. In Section
4, we prove  Theorem \ref{thm1.1}. Finally, in Section 5, we give
the generalization to random multivariate  trigonometric
polynomials.

\section{General formulation}

Let $\mathbb{S}^{N-1}:=\{\x\in \mathbb{R}^{N}:|\x|_{2}=1\}$ denote
the unit sphere of $\mathbb{R}^{N}$ equipped with the usual
surface  measure $d\sigma$, where $\x\cdot\y$ is the usual
Euclidean  inner product and $|\x|_{2}=(\x\cdot{\x})^{1/2}$ is the
usual Euclidean norm. Suppose that $K\subset\mathbb{R}^{d}$ is
compact and $\mu$ is a Borel probability measure supported on $K$.
We denote by $L_{p}:=L_{p}(K,d\mu),1\leq p<\infty$, the Lebesgue
space on $K$ endowed with the finite norm
$$\|f\|_{p}:=\Big(\int_{K}|f(\x)|^{p}d\mu(\x)\Big)^{1/p},$$
and by $L_{\infty}=C(K)$ the space of continuous functions with
the norm
$$\|f\|_{\infty}:=\max\limits_{\x\in K}|f(\x)|.$$Then, $L_2$ is a
Hilbert space with the inner product
\begin{equation}\label{1.00}\langle f,g\rangle=\int_K
f(\x)g(\x)d\mu(\x).\end{equation}

Suppose that  $\mathcal{V}_{n}$ is the $N$-dimensional linear
subspace of $C(K)$, and
$\{\varphi_{1},\varphi_{2},\dots,\varphi_{N}\}$ is an arbitrary
orthonormal basis for $\mathcal{V}_{n}$ with respect to the inner
product \eqref{1.00}. For $\a=(a_1,\dots,a_N)\in\Bbb R^N$, we
define
$$f_{\a}(\x)=\sum_{j=1}^N a_j\varphi_j(\x).$$

In this paper we study Nikolskii factors. For given $1\leq
p,q\leq\infty$, the $(p,q)$-Nikolskii factor of
$f_{\a}\in\mathcal{V}_{n}$ is defined by
$$\mathcal{N}_{p,q}(f_{\a}):=\frac{\|f_{\a}\|_{q}}{\|f_{\a}\|_{p}}.$$
The worst case $(p,q)$-Nikolskii factor for $\mathcal{V}_{n}$ is
defined by
$$N^{\rm{wor}}_{p,q}(\mathcal{V}_{n}):=\sup_{0\neq f_{\a}\in
\mathcal{V}_{n}}\mathcal{N}_{p,q}(f_{\a})=
\sup_{\a\neq0}\frac{\|f_{\a}\|_{q}}{\|f_{\a}\|_{p}},$$ where the
supremum is taken over all functions $f_{\a}$ of $\mathcal{V}_{n}$
whose coefficients $\a$ are not all zero.

A random  function $f_{\a}\in\mathcal{V}_{n}$ is defined by
$$f_{\a}(\x)=\sum\limits^{N}_{i=1}a_{i}\varphi_{i}(\x),$$
where $\a=(a_{1},a_{2},...,a_{N})\in \mathbb{R}^{N}$, the
coefficients $a_i$ are independent $N(0,\sigma^2)$ random
variables with the common normal density functions $$\frac1
{\sqrt{2\pi}\sz}e^{-\frac{a_i^2}{2\sz^2}}.$$ That is,  $\a\sim
N(0, \sz^2I_N)$ is a $\mathbb{R}^{N}$-valued Gaussian random
vector with mean $0$ and covariance matrix $\sigma^{2}I_{N}$,
where $I_N$ is the $N$ by $N$ identity matrix. Let $\gamma_{N}$
denote the corresponding Gaussian measure on $\mathbb{R}^{N}$
  given by
$$\gamma_{N}(G)=\frac{1}{(2\pi\sigma^{2})^{N/2}}
\int_{G}e^{-\frac{|\x|_{2}^{2}}{2\sigma^{2}}}d\x\ {\rm for\
each\ Borel\ subset}\  G\subset\mathbb{R}^{N}.$$

For given $1\le p,q\le \infty$, the average (expected) Nikolskii
factor for $\mathcal{V}_{n}$ is defined by
\begin{align}\label{1.02}
N^{{\rm ave}}_{p,q}(\mathcal{V}_{n}):=\Bbb
E\frac{\|f_{\a}\|_{q}}{\|f_{\a}\|_{p}}=\int_{\mathbb{R}^{N}}
\frac{\|f_{\a}\|_{q}}{\|f_{\a}\|_{p}}d\gamma_{N}(\a)=\frac{1}{(2\pi\sigma^{2})^{N/2}}
\int_{\Bbb R^N}
\frac{\|f_{\a}\|_{q}}{\|f_{\a}\|_{p}}e^{-\frac{|\a|_{2}^{2}}{2\sigma^{2}}}
d\a.
\end{align}

Note that the average  Nikolskii factor $N^{{\rm
ave}}_{p,q}(\mathcal{V}_{n})$ is independent of the choice of the
orthonormal basis for $\mathcal{V}_{n}$ due to  the rotation
invariance of the Gaussian measure $\gamma_{N}$.

A connection between the average Nikolskii factor $N^{{\rm
ave}}_{p,q}(\mathcal{V}_{n})$ and the worst case Nikolskii factor
$N^{\rm{wor}}_{p,q}(\mathcal{V}_{n})$ can be seen as follows:
\begin{align*}
N^{{\rm ave}}_{p,q}(\mathcal{V}_{n})=\Bbb
E\frac{\|f_{\a}\|_{q}}{\|f_{\a}\|_{p}} \leq\sup_{\a\neq 0}
\frac{\|f_{\a}\|_{q}}{\|f_{\a}\|_{p}}=N^{\rm{wor}}_{p,q}(\mathcal{V}_{n}).
\end{align*}

We have the following proposition.

\begin{prop}
For any $1\leq p,q\leq\infty$, we have
$$N^{\rm{ave}}_{p,q}(\mathcal{V}_{n})=
\mathbb{E}_{\a}\frac{\|f_{\a}\|_{q}}{\|f_{\a}\|_{p}}
=\mathbb{E}_{\b}\frac{\|f_{\b}\|_{q}}{\|f_{\b}\|_{p}},$$
where $\b\sim N(0,I_{N})_{\mathbb{R}^{N}}$ is a
$\mathbb{R}^{N}$-valued Gaussian random vector with mean $0$ and
covariance matrix $I_{N}$.
\end{prop}
\begin{proof}
Let $\a=\sigma\b$.  Due to the fact  $\a\sim N(0,\sz I_N)$, we
have $\b\sim N(0,I_N)$, and
\begin{align*}
\mathbb{E}_{\a}\frac{\|f_{\a}\|_{q}}{\|f_{\a}\|_{p}}&=
\frac{1}{(2\pi\sigma^{2})^{N/2}} \int_{\mathbb{R}^{N}}
\frac{\|f_{\a}\|_{q}}{\|f_{\a}\|_{p}}e^{-\frac{|\a|_{2}^{2}}{2\sigma^{2}}}d\a
\\&=\frac{1}{(2\pi\sigma^{2})^{N/2}}
\int_{\mathbb{R}^{N}}\frac{\|f_{\sigma\b}\|_{q}}{\|f_{\sigma\b}\|_{p}}
e^{-\frac{|\sigma\b|_{2}^{2}}{2\sigma^{2}}}d(\sigma\b)\\
&=\frac{1}{(2\pi)^{N/2}}\int_{\mathbb{R}^{N}}
\frac{\|f_{\b}\|_{q}}{\|f_{\b}\|_{p}}e^{-\frac{|\b|_{2}^{2}}2}d\b
\\&=\mathbb{E}_{\b}\frac{\|f_{\b}\|_{q}}{\|f_{\b}\|_{p}}.
\end{align*}
This completes the proof.
\end{proof}

In the sequel, by Proposition 2.1 we always assume that $\sz=1$
and $\a\sim N(0,
 I_N)$ without loss of generality.

\begin{thm}\label{thm2.2}
For any $1\leq p,q\leq\infty,\,\,k,l\in\mathbb{N}$, we have for
$l<k+N$,
\begin{equation}\label{2.3}
\mathbb{E}\frac{\|f_{\a}\|^{k}_{q}}{\|f_{\a}\|^{l}_{2}}\asymp
N^{-l/2}\mathbb{E}\|f_{\a}\|^{k}_{q},
\end{equation}
and for $l<N$,
\begin{equation}\label{2.4}
\mathbb{E}\frac{\|f_{\a}\|^{k}_{2}}{\|f_{\a}\|^{l}_{p}}\asymp
N^{k/2}\mathbb{E}\frac{1}{\|f_{\a}\|^{l}_{p}}.
\end{equation}
In particular, we have
\begin{equation}\label{2.5}
\mathbb{E}\frac{\|f_{\a}\|_{q}}{\|f_{\a}\|_{2}}\asymp
N^{-1/2}\mathbb{E}\|f_{\a}\|_{q}.
\end{equation}
\end{thm}
\begin{proof}
Using Parseval's identity $\|f_{\a}\|_2=|\a|_2$ and
$\a=r\xi,r=|\a|,\xi\in\mathbb{S}^{N-1}$, we have for $l<N+k$,
\begin{align*}
\mathbb{E}\frac{\|f_{\a}\|^{k}_{q}}{\|f_{\a}\|^{l}_{2}}&
=\frac{1}{(2\pi)^{N/2}}\int_{\mathbb{R}^{N}}
\frac{\|f_{\a}\|^{k}_{q}}{\|f_{\a}\|^{l}_{2}}e^{-\frac{|\a|_{2}^{2}}{2}}d\a\\
&=\frac{1}{(2\pi)^{N/2}}\int_{\mathbb{R}^{N}}
\frac{\|f_{\a}\|^{k}_{q}}{|\a|_{2}^{l}}e^{-\frac{|\a|_{2}^{2}}{2}}d\a\\
&=\frac{1}{(2\pi)^{N/2}}\int^{\infty}_{0}e^{-\frac{r^{2}}{2}}
r^{k-l+N-1}dr
\int_{\mathbb{S}^{N-1}}\|f_{\xi}\|^{k}_{q}d\sigma(\xi)\\
&=\frac{(\sqrt{2})^{k-l-2}}{\pi^{N/2}}\Gamma\left(\frac{k-l+N}{2}\right)
\int_{\mathbb{S}^{N-1}}\|f_{\xi}\|^{k}_{q}d\sigma(\xi),
\end{align*}
and
\begin{align*}
\mathbb{E}\|f_{\a}\|^{k}_{q}&=\frac{1}{(2\pi)^{N/2}}
\int_{\mathbb{R}^{N}}\|f_{\a}\|^{k}_{q}e^{-\frac{|\a|_{2}^{2}}{2}}d\a\\
&=\frac{1}{(2\pi)^{N/2}}\int_{\mathbb{S}^{N-1}}
\|f_{\xi}\|^{k}_{q}d\sigma(\xi)
\int^{\infty}_{0}e^{-\frac{r^{2}}{2}}r^{k+N-1}dr\\
&=\frac{(\sqrt{2})^{k-2}}{\pi^{N/2}}\Gamma\left(\frac{k+N}{2}\right)
\int_{\mathbb{S}^{N-1}}\|f_{\xi}\|^{k}_{q}d\sigma(\xi).
\end{align*}
It follows that
\begin{equation}\label{2.6}\mathbb{E}\frac{\|f_{\a}\|^{k}_{q}}{\|f_{\a}\|^{l}_{2}}=
(\sqrt{2})^{-l}\Gamma\left(\frac{k-l+N}{2}\right)\Big/
\Gamma\left(\frac{k+N}{2}\right)
\mathbb{E}\|f_{\a}\|^{k}_{q}.\end{equation} By Stiring's formula
(see \cite[p. 18]{AAR}), we have the asymptotic estimate
\begin{equation}\label{2.7}
\lim\limits_{x\rightarrow+\infty}\frac{\Gamma(x+1)
e^{x}}{\sqrt{2\pi}x^{x+\frac{1}{2}}}=1,
\end{equation}
from which it follows that
$$
\Gamma\left(\frac{k-l+N}{2}\right)\Big/
\Gamma\left(\frac{k+N}{2}\right) \asymp N^{-l/2}.
$$
By \eqref{2.6} we obtain \eqref{2.3}.

Similarly,  we have for $l<N+k$,
\begin{align*}
\mathbb{E}\frac{\|f_{\a}\|^{k}_{2}}{\|f_{\a}\|^{l}_{p}}
 &=\frac{1}{(2\pi)^{N/2}}\int_{\mathbb{R}^{N}}
\frac{|\a|_{2}^{k}}{\|f_{\a}\|_{p}^{l}}e^{-\frac{|\a|_{2}^{2}}{2}}d\a\\
&=\frac{1}{(2\pi)^{N/2}}\int^{\infty}_{0}
e^{-\frac{r^{2}}{2}}r^{k-l+N-1}dr
\int_{\mathbb{S}^{N-1}}\frac{1}{\|f_{\xi}\|^{l}_{p}}d\sigma(\xi)\\
&=\frac{(\sqrt{2})^{k-l-2}}{\pi^{N/2}}\Gamma\left(\frac{k-l+N}{2}\right)
\int_{\mathbb{S}^{N-1}}\frac{1}{\|f_{\xi}\|^{l}_{p}}d\sigma(\xi),
\end{align*}and for $l<N$,
\begin{align*}
\mathbb{E}\frac{1}{\|f_{\a}\|^{l}_{p}}
&=\frac{1}{(2\pi)^{N/2}}\int_{\mathbb{S}^{N-1}}
\frac{1}{\|f_{\xi}\|^{l}_{p}}d\sigma(\xi)
\int^{\infty}_{0}e^{-\frac{r^{2}}{2}}r^{-l+N-1}dr\\
&=\frac{(\sqrt{2})^{-l-2}}{\pi^{N/2}}\Gamma\left(\frac{N-l}{2}\right)
\int_{\mathbb{S}^{N-1}}\frac{1}{\|f_{\xi}\|^{l}_{p}}d\sigma(\xi).
\end{align*}
It follows that for $l<N$,
$$\mathbb{E}\frac{\|f_{\a}\|^{k}_{2}}{\|f_{\a}\|^{l}_{p}}=
(\sqrt{2})^{k}\Gamma\left(\frac{k-l+N}{2}\right)\Big/
\Gamma\left(\frac{N-l}{2}\right)
\mathbb{E}\frac1{\|f_{\a}\|^{l}_{p}}.$$ By \eqref{2.7}
we get
$$\mathbb{E}\frac{\|f_{\a}\|^{k}_{2}}{\|f_{\a}\|^{l}_{p}}\asymp N^{k/2} \mathbb{E}\frac1{\|f_{\a}\|^{l}_{p}}.$$
This completes the proof of Theorem \ref{thm2.2}.
\end{proof}

Next, we estimate $\mathbb{E}\|f_{\a}\|_{q}^q$. This quantity is
intimately related to the Christoffel function for $\mathcal V_n$.
We recall that the Christoffel function for $\mathcal V_n$ is
defined by
$$\Lz(\x)=\inf_{p(\x)=1,\,p\in\mathcal V_n}\|p\|_2^2.$$It follows
from \cite[Theorem 3.5.6]{DX} that
$$\Lz(\x)=\frac1{\sum_{k=1}^N\varphi_k^2(\x)}.$$
We have the following  result.

\begin{thm}\label{thm2.3}
For any $1\leq q<\infty$. We have
\begin{equation}\label{2.8}
\mathbb{E}\|f_{\a}\|^{q}_{q}=C(q)^{q}\Big\|\Big(\sum\limits^{N}_{i=1}
|\varphi_{i}|^{2}\Big)^{1/2}\Big\|_{q}^q\ ,
\end{equation}
where $\a\sim N(0,I_N)$, and $$C(q)=(\Bbb
E|a_1|^q)^{1/q}=\Big(\frac1{\sqrt{2\pi}}\int_{\Bbb
R}|t|^qe^{-t^2/2}dt\Big)^{1/q}=\pi^{-\frac{1}{2q}}
2^{\frac{1}{2}}\Gamma^{1/q}\Big(\frac{q+1}{2}\Big).$$
\end{thm}

\begin{proof}

For $1\leq q<\infty$, by the Fubini theorem, we have
\allowdisplaybreaks
\begin{align}\label{2.9}
\mathbb{E}\|f_{\a}\|^{q}_{q}&
=\mathbb{E}\int_{K}|f_{\a}(\x)|^{q}d\mu(\x)\nonumber\\
&=\int_{K}|m(\x)|^{q/2}
\mathbb{E}\Big|\frac{f_{\a}(\x)}{\sqrt{m(\x)}}\Big|^{q}d\mu(\x),
\end{align}where
$$m(\x):=\frac1{\Lz(\x)}=\sum\limits^{N}_{k=1}|\varphi_{k}(\x)|^{2}.$$
Since $\a\sim N(0,I_N)$, we get for arbitrary fixed $\x\in K$,
$$
\eta(\x):=\frac{1}{\sqrt{m(\x)}}f_{\a}(\x)\sim N(0,1),
$$and $$\Bbb E |\eta(\x)|^q=\Bbb
E|a_1|^q=C(q)^q=\frac1{\sqrt{2\pi}}\int_{\Bbb
R}|t|^qe^{-t^2/2}dt=\pi^{-\frac{1}{2}}2^{\frac{q}{2}}\Gamma\Big(\frac{q+1}{2}\Big),$$which
is independent of $\x\in K$. It follows from \eqref{2.9} that
\begin{align*}\label{Ex}
\mathbb{E}\|f_{\a}\|^{q}_{q}&=\int_{K}|m(\x)|^{q}\mathbb{E}|\eta(\x)|^qd\mu(\x)\notag
\\&=C(q)^{q}\int_{K}|m(\x)|^{q/2}d\mu(\x)\notag\\& = C(q)^{q}
\Big\|\Big(\sum\limits^{N}_{k=1}|\varphi_{i}|^{2}
\Big)^{1/2}\Big\|^{q}_{q}\,.
\end{align*}
This completes the proof of Theorem \ref{thm2.3}.
\end{proof}

\begin{cor}\label{Cor2.4}
For any $1\leq s, q<\infty$, we have
\begin{equation}\label{2.10}\mathbb{E}\|f_{\a}\|^{s}_{q}\asymp
\Big\|\Big(\sum\limits^{N}_{i=1}
|\varphi_{i}|^{2}\Big)^{1/2}\Big\|^{s}_{q}.\end{equation}
\end{cor}

\begin{proof} It follows from \cite[Theorem 1]{WZ} that
$$\left(\mathbb{E}\|f_{\a}\|^{s}_{q}\right)^{1/s}\asymp
\left(\mathbb{E}\|f_{\a}\|^{q}_{q}\right)^{1/q}.$$According to
Theorem \ref{thm2.3}, we get the required result \eqref{2.10}.
\end{proof}

\begin{rem}For $1\le s, q<\infty$ and  a Rademacher random function, we have the Kahane-Khintchine
inequality (see \cite{LT})
$$\mathbb{E}\Big(\big\|\sum\limits^{N}_{i=1}\vz_{i}\varphi_{i}\big\|^s_q\Big)\asymp
\Big\|\Big(\sum\limits^{N}_{i=1}
|\varphi_{i}|^{2}\Big)^{1/2}\Big\|^{s}_{q}. $$\end{rem}

\section{Average Nikolskii factor for $\mathcal{T}_{n}$}

In this and next sections, we always suppose that
$$
T_{\a}(x)=a_1+\sum_{k=1}^n(a_{2k}\sqrt2\cos kx+a_{2k+1}\sqrt 2\sin
kx),
$$where $\a\sim N(0,I_N)$. We note that $\{1, \sqrt2 \cos x,
\sqrt2\sin x,\dots, \sqrt2 \cos nx, \sqrt2\sin nx\}$ is an
orthonormal basis for $\mathcal T_n$.

\subsection{Estimates of $N^{\rm{ave}}_{2,q}(\mathcal{T}_{n})$}

\begin{thm}\label{thm3.1}
Let $p=2, \,1\leq q\leq\infty$. We have
\begin{equation}\label{3.1}N^{\rm{ave}}_{2,q}(\mathcal{T}_{n})=
\mathbb{E}\Big(\frac{\|T_{\a}\|_{q}}{\|T_{\a}\|_{2}}\Big)\asymp
N^{-1/2}\Bbb E\|T_{\a}\|_q\asymp\left\{
\begin{aligned}
&1,&\ \,1\leq q<\infty,\\
&\sqrt{\ln N},&\ \,q=\infty,
\end{aligned}
\right.\end{equation} where $N=2n+1$.
\end{thm}
\begin{proof} We note that $\{1, \sqrt2\cos kx, \sqrt2\sin kx,\ k=1,2,\dots,n\}$ is an
orthonormal basis for $\mathcal T_n$, and $$m(x):=1^2+\sum_{k=1}^n
(2\cos^2 kx+2\sin^2 kx)=2n+1=N.$$ For $1\leq s,q<\infty$,
\eqref{2.10} and (\ref{2.5}) give
$$\mathbb{E}\|T_{\a}\|^{s}_{q}\asymp \|m(x)^{1/2}\|_q^s\asymp N^{s/2},$$
and $$ N^{\rm{ave}}_{2,q}(\mathcal{T}_{n})\asymp N^{-1/2}\Bbb
E\|T_{\a}\|_q\asymp 1.
$$

For $q=\infty$, by the Nikolskii inequality (\ref{1.2}), the
H\"{o}lder inequality,  and Theorem \ref{thm2.3}, we have for
$2<q<\infty$,
\begin{align*}
\mathbb{E}\|T_{\a}\|_{\infty}&\ll N^{1/q}\mathbb{E}\|T_{\a}\|_{q}\leq N^{1/q}(\mathbb{E}\|T_{\a}\|^{q}_{q})^{1/q} \\
&=N^{1/q+1/2}C(q)\ll N^{1/2+1/q}q^{1/2},
\end{align*}
where in the last inequality, we used the fact that $$
C(q)^{1/q}\ll q^{1/2}.$$This is due to
$$\lim\limits_{q\rightarrow+\infty}\frac{C(q)}{q^{1/2}}=
\lim\limits_{q\rightarrow+\infty}\Big(\frac{q-1}{q}\Big)^{1/2}
2^{\frac{1}{2q}}\exp\Big(-\frac{q-1}{2q}\Big)=e^{-1/2}.$$We remark
that the  constants in the above inequalities are independent of
$N$ and $q$. Thus, taking $q=\ln N$, we get
$$\mathbb{E}\|T_{\a}\|_{\infty}\ll (N\ln N)^{1/2}.$$ By
(\ref{2.5}) we obtain
\begin{equation}N^{\rm{ave}}_{2,\infty}(\mathcal{T}_{n})\ll\sqrt{\ln
N}.\label{3.1-0}\end{equation}

 Now we show the lower bound  of
 $N^{\rm{ave}}_{2,\infty}(\mathcal{T}_{n})$. We set
 $$X=(X_1,\dots, X_N),\ \ X_k=\frac1{\sqrt N}T_{\a}(x_k), \ x_k=\frac{2\pi k}{2n+1}, \ 1\le k\le
 N,$$where $T_{\a}$ is the random trigonometric polynomial given by \eqref{1.1},
 $\a\sim N(0,I_N)$. Then $X$ is the Gaussian centered
 random vector with covariance matrix $C=(C_{ij})_{i,j=1}^N$,
 where $C_{ij}=\Bbb E(X_iX_j)$. We note that $$\Bbb E(a_k
 a_l)=\delta_{k,l}=\left\{
\begin{aligned}
 &0,&\ \,l\neq s,\\
 &1,&\ \,l=s.
\end{aligned}
\right., \ \ 1\le k,l\le N,$$and
$$T_{\a}(x_i)=a_1+\sum_{k=1}^n(a_{2k}\sqrt2\cos kx_i+a_{2k+1}\sqrt 2\sin kx_i).$$ It follows that
\begin{align*}\Bbb E(X_iX_j)&=\frac1N\Bbb E(T_{\a}(x_i)T_{\a}(x_j))=\frac1N(1+2\sum_{k=1}^n(\cos x_i\cos x_j+\sin x_i\sin x_j))\\
&=\frac 1N(1+2\sum_{k=1}^n\cos k(x_i-x_j))=\frac
1ND_n(x_i-x_j)=\delta_{i,j},\end{align*}where $D_n$ is the
Dirichlet kernel,
$$D_n(x)=1+2\sum_{k=1}^n\cos kx= \frac{\sin(n+\frac12)x}{\sin \frac
x2}.$$ Hence, $X\sim N(0,I_N)$. This means that $X_1,\dots, X_N$
are the i.i.d. random variables with $X_i\sim N(0,1)$. It follows
from \cite{V, P} that
$$\Bbb E \|T_{\a}\|_\infty\ge \sqrt N \ \Bbb E\max_{1\le k\le
N}|X_k|\asymp \sqrt N\, \sqrt{\ln N}.
$$
By \eqref{2.5} and \eqref{3.1-0} we obtain
$$N^{\rm{ave}}_{2,\infty}(\mathcal{T}_{n})\asymp \sqrt{\ln
N}.$$Theorem \ref{thm3.1} is proved. \end{proof}

\subsection{Upper estimates of $\mathbb{E}(\frac{1}{\|T_{\a}\|^{r}_{\infty}})$}
\begin{lem}(See \cite[p. 2]{B}, \cite{DW, M})\label{lem3.3}
If $X\sim N(0,1)$, then for all $t>0$,
$$\sqrt{2}\pi^{-1/2}(t^{-1}-t^{-3})e^{-\frac{t^{2}}{2}}\leq\mathbb{P}(|X|\geq t)\geq \sqrt{2}\pi^{-1/2}t^{-1}e^{-\frac{t^{2}}{2}}.$$
In particular, if $\delta\in(0,e^{-1}]$, then
$\mathbb{P}(|X|\geq c_{0}\sqrt{\ln \delta^{-1}})>\delta$
for some absolute constant $c_{0}>0$.
\end{lem}

\begin{thm} \label{thm3.3}
Let $r>0$. Then for $N> \max\{e^4, r+1\}$, we have
\begin{equation}\label{3.2}
\mathbb{E}\frac{1}{\|T_{\a}\|^{r}_{\infty}}\ll N^{-r/2}(\ln
N)^{-r/2}.
\end{equation}
\end{thm}
\begin{proof} As in the proof of Theorem \ref{thm3.1},
we set
 $$X=(X_1,\dots, X_N),\ \ X_k=\frac1{\sqrt N}T_{\a}(x_k), \ x_k=\frac{2\pi k}{N}, \ 1\le k\le
 N.$$ It follows  that $X\sim N(0,I_{N})$. Therefore,
\begin{align*}
\mathbb{E}\frac{1}{\|T_{\a}\|^{r}_{\infty}}
&\leq\mathbb{E}\frac{1}{\max\limits_{1\leq k\leq
N}|T_{\a}(x_{k})|^{r}}=N^{-\frac{r}{2}}\,\mathbb{E}\frac{1}{\max\limits_{1\leq
k\leq N} |X_{k}|^{r}}.
\end{align*}

\noindent By the variable substitution $t=s^{-r}$, we have
\begin{align*}
\mathbb{E}\frac{1}{\max\limits_{1\leq k\leq N}|X_{k}|^{r}}
&=\int^{\infty}_{0}\mathbb{P}\Big(\frac{1}{\max\limits_{1\leq k\leq N}
|X_{k}|^{r}}\geq t\Big)dt\\
&=\int^{\infty}_{0}\mathbb{P}\Big(\max\limits_{1\leq k\leq N}
|X_{k}|\leq t^{-\frac{1}{r}}\Big)dt\\
&=\int^{\infty}_{0}(\mathbb{P}(|X_{1}|\leq t^{-\frac{1}{r}}))^{N}dt\\
&=r\int^{\infty}_{0}(\mathbb{P}(|X_{1}|\leq s))^{N}\frac{ds}{s^{r+1}}\\
&\asymp\Big(\int^{\frac{c_{0}}{2}(\ln N)^{1/2}}_{0}+\int^{\infty}_{\frac{c_{0}}{2}(\ln N)^{1/2}}\Big)(\mathbb{P}(|X_{1}|\leq s))^{N}\frac{ds}{s^{r+1}}\\
&=I+II,
\end{align*}where $c_0$ is the absolute constant given in Lemma
\ref{lem3.3}. By Lemma \ref{lem3.3} with $\delta_0=N^{-\frac1{4}}$ and
$N>e^4$, we get
\begin{equation}\label{3.2-2}\Bbb P(|X_1|\le \frac{c_{0}}{2}(\ln N)^{1/2})=1-\Bbb P(|X_1|>c_0\sqrt{\ln \delta_0^{-1}})\le 1-\delta_0=1-N^{-1/4}.\end{equation}
We note that
$$\lim\limits_{s\rightarrow0^{+}}\frac{\mathbb{P}(|X_{1}|\leq
s)}{s}=\lim\limits_{s\rightarrow0^{+}}\sqrt{\frac{2}{\pi}}
\frac{\int^{s}_{0}e^{-\frac{t^{2}}{2}}dt}{s}=\sqrt{\frac{2}{\pi}},$$
and $$\mathbb{P}(|X_{1}|\leq s)\leq1.$$ This means that
$$ \mathbb{P}(|X_{1}|\leq
s)\ll s.$$ By \eqref{3.2-2} we obtain
\begin{align*}
I&=\int^{\frac{c_{0}}{2}(\ln N)^{1/2}}_{0}(\mathbb{P}
(|X_{1}|\leq s))^{N}\frac{1}{s^{r+1}}ds\\
&\ll\int^{\frac{c_{0}}{2}(\ln N)^{1/2}}_{0}(\mathbb{P}
(|X_{1}|\leq s))^{N-r-1}ds\\
&\ll(\ln N)^{1/2}(\mathbb{P}(|X_{1}|\leq
\frac{c_{0}}{2}(\ln N)^{1/2}))^{N-r-1}\\
&\leq(\ln N)^{1/2}\big(1-N^{-1/4}\big)^{N-r-1}.
\end{align*}
Since $$\lim_{N\to\infty}(\ln
N)^{\frac{r+1}2}\big(1-N^{-1/4}\big)^{N-r-1}=\lim_{N\to\infty}(\ln
N)^{\frac{r+1}2}e^{-(N-r-1)N^{-1/4}}=0,$$we get $$(\ln
N)^{\frac{r+1}2}\big(1-N^{-1/4}\big)^{N-r-1}\ll 1.$$ It follows
that
$$I\ll (\ln N)^{1/2}\big(1-N^{-1/4}\big)^{N-r-1}\ll (\ln
N)^{-r/2}.$$ We also have
\begin{align*}
II&=\int^{\infty}_{\frac{c_{0}}{2}(\ln N)^{1/2}}
(\mathbb{P}(|X_{1}|\leq
s))^{N}\frac{1}{s^{r+1}}ds\leq\int^{\infty}_{\frac{c_{0}}{2}(\ln
N)^{1/2}}\frac{1}{s^{r+1}}ds\asymp(\ln N)^{-\frac{r}{2}}.
\end{align*}
Hence, we obtain
\begin{align*}
\mathbb{E}\frac{1}{\max\limits_{1\leq k\leq N}
|X_{k}|^{r}}&\ll(\ln N)^{-\frac{r}{2}}.
\end{align*}
This deduces that
$$
\mathbb{E}\frac{1}{\|T_{\a}\|^{r}_{\infty}}\ll
N^{-\frac{r}{2}}(\ln N)^{-\frac{r}{2}},
$$
which completes the proof of Theorem \ref{thm3.3}.
\end{proof}

\section {proof of Theorem \ref{thm1.1}}

\begin{proof} {\bf Upper bounds}.

Let $1\le p< q\le \infty$.
 By the  H\"{o}lder inequality, we have
$$\|T_{\a}\|^{2}_{2}\leq\|T_{\a}\|^{\frac{1}{2}}_{1}
\|T_{\a}\|^{\frac{3}{2}}_{3},$$
which implies that
\begin{align*}
\mathbb{E}\frac{\|T_{\a}\|_{q}}{\|T_{\a}\|_{1}}&\leq\mathbb{E}
\Big(\frac{\|T_{\a}\|_{q}\|T_{\a}\|^{3}_{3}}{\|T_{\a}\|^{4}_{2}}\Big)\\
&=\mathbb{E}\Big(\frac{\|T_{\a}\|_{q}}{\|T_{\a}\|^{2}_{2}}
\frac{\|T_{\a}\|^{3}_{3}}{\|T_{\a}\|^{2}_{2}}\Big)\\
&\leq\Big(\mathbb{E}\frac{\|T_{\a}\|^{2}_{q}}
{\|T_{\a}\|^{4}_{2}}
\Big)^{1/2}\Big(\mathbb{E}\frac{\|T_{\a}\|^{6}_{3}}
{\|T_{\a}\|^{4}_{2}}\Big)^{1/2}.
\end{align*}
By (\ref{2.3}), we deduce that
$$
\mathbb{E}\frac{\|T_{\a}\|^{2}_{q}}{\|T_{\a}\|^{4}_{2}}\asymp
N^{-2}\mathbb{E}\|T_{\a}\|^{2}_{q},
$$
and
$$\mathbb{E}\frac{\|T_{\a}\|^{6}_{3}}{\|T_{\a}\|^{4}_{2}}
\asymp N^{-2}\mathbb{E}\|T_{\a}\|^{6}_{3}.
$$
 According to  \cite[Theorem 1]{WZ} and (\ref{3.1}), we have for
 $s\ge1$,
\begin{equation}\label{4.1-0}\mathbb{E}\|T_{\a}\|_{q}^{s}\asymp \left\{
\begin{aligned}
 &N^{s/2},&\ \,1\leq q<\infty,\\
 &N^{s/2}(\ln N)^{s/2},&\ \, q=\infty.
\end{aligned}
\right.\end{equation} It follows  that
\begin{equation}\label{4.1}
\mathbb{E}\frac{\|T_{\a}\|_{q}}{\|T_{\a}\|_{p}}\leq\mathbb{E}\frac{\|T_{\a}\|_{q}}{\|T_{\a}\|_{1}}\ll\left\{
\begin{aligned}
 &1,&\ \,1\leq p\leq q<\infty,\\
 &(\ln N)^{1/2},&\ \,1\leq p<\infty, q=\infty.
\end{aligned}
\right.
\end{equation}

Let $1\leq q< p\leq\infty$.  For $p<\infty$,  we have
$$\mathbb{E}\frac{\|T_{\a}\|_{q}}{\|T_{\a}\|_{p}}\le1.$$
For $p=\infty$, by the Cauchy inequality we have
\begin{align*}
\mathbb{E}\frac{\|T_{\a}\|_{q}}{\|T_{\a}\|_{\infty}}&
=\mathbb{E}\Big(\frac{\|T_{\a}\|_{q}}{\|T_{\a}\|_{2}}\frac{\|T_{\a}\|_{2}}{\|T_{\a}\|_{\infty}}\Big)
\leq\Big(\mathbb{E}\frac{\|T_{\a}\|^{2}_{q}}{\|T_{\a}\|^{2}_{2}}\Big)^{1/2}
\Big(\mathbb{E}\frac{\|T_{\a}\|^{2}_{2}}{\|T_{\a}\|^{2}_{\infty}}\Big)^{1/2}.
\end{align*}
It follows from (\ref{2.4}) and (\ref{3.2}) that
\begin{equation}\label{4.5}
\mathbb{E}\frac{\|T_{\a}\|^{2}_{2}}{\|T_{\a}\|^{2}_{\infty}}\asymp
N \, \mathbb{E}\frac{1}{\|T_{\a}\|^{2}_{\infty}}\ll (\ln N)^{-1}.
\end{equation}
By (\ref{2.3}) and \eqref{4.1-0} we obtain
\begin{equation*}
\mathbb{E}\frac{\|T_{\a}\|^{2}_{q}}{\|T_{\a}\|^{2}_{2}}\asymp
N^{-1}\,\mathbb{E}\|T_{\a}\|^{2}_{q}\asymp1,
\end{equation*}
which, combining with (\ref{4.5}),  show that
$$\mathbb{E}\frac{\|T_{\a}\|_{q}}{\|T_{\a}\|_{\infty}}\ll
\frac{1}{\sqrt{\ln N}}.$$

Hence, we obtain the upper bounds of
$N^{\rm{ave}}_{p,q}(\mathcal{T}_{n})$ for $1\le p,q\le \infty$ as
follows.
\begin{equation}\label{4.6}
N^{\rm{ave}}_{p,q}(\mathcal{T}_{n})\ll \left\{
\begin{aligned}
 &1,&\ \,1\leq p,q<\infty\ or\ p=q=\infty,\\
 &(\ln N)^{1/2},&\ \,1\leq p<\infty,q=\infty,\\
 &(\ln N)^{-1/2},&\ \,1\leq q<\infty,p=\infty.
\end{aligned}
\right.
\end{equation}

{\bf Lower bounds}.

By the Cauchy inequality, we have
\begin{align*}
1=\mathbb{E}(1)&=\mathbb{E}\sqrt{\frac{\|T_{\a}\|_{q}}{\|T_{\a}\|_{p}}
\frac{\|T_{\a}\|_{p}}{\|T_{\a}\|_{q}}}\leq\Big(\mathbb{E}\frac{\|T_{\a}\|_{q}}{\|T_{\a}\|_{p}}\Big)^{1/2}
\Big(\mathbb{E}\frac{\|T_{\a}\|_{p}}{\|T_{\a}\|_{q}}\Big)^{1/2},
\end{align*}
which deduces that
\begin{equation}\label{4.7}
\mathbb{E}\frac{\|T_{\a}\|_{q}}{\|T_{\a}\|_{p}}\geq
\frac{1}{\mathbb{E}\frac{\|T_{\a}\|_{p}}{\|T_{\a}\|_{q}}}.
\end{equation} By \eqref{4.6} and \eqref{4.7} we get
$$
N^{\rm{ave}}_{p,q}(\mathcal{T}_{n})\gg \left\{
\begin{aligned}
 &1,&\ \,1\leq p,q<\infty\ or\ p=q=\infty,\\
 &(\ln N)^{1/2},&\ \,1\leq p<\infty,q=\infty,\\
 &(\ln N)^{-1/2},&\ \,1\leq q<\infty,p=\infty,
\end{aligned}
\right.
$$
which gives the lower bounds of
$N^{\rm{ave}}_{p,q}(\mathcal{T}_{n})$ for $1\le p,q\le \infty$.

Theorem \ref{thm1.1} is proved.\end{proof}

\begin{rem} In the proof of Theorem \ref{thm1.1} with $q=\infty$ or $p=\infty$, it suffices to give the upper bounds
 of $\mathbb{E} \big(\|T_{\a}\|^{r}_{\infty}\big)$ and $\mathbb{E}\big(\frac{1}{\|T_{\a}\|^{r}_{\infty}}\big)$.  \end{rem}

\section{Generalization to $\Bbb T^d, \ d>1$}

In this section, we  shall extend the  results of the average
Nikolskii factors for random trigonometric polynomials on
1-dimensional torus $\Bbb T$ to the ones on the $d$-dimensional
torus $\Bbb T^d=[0,2\pi]^d$.

Let $L_p(\Bbb T^{d})\ (1\le p<\infty)$ be the usual Lebesgue space
consisting of measurable functions $f$ with finite norm
$$\|f\|_p=\Big(\frac1{(2\pi)^{d}}\int_{\Bbb T^{d}}|f(\x)|^pdx\Big)^{1/p},$$
and $L_\infty(\Bbb T^{d})$ be the space $C(\Bbb T^d)$ of
continuous functions on $\Bbb T^{d}$ with norm
$\|f\|_\infty=\max\limits_{\x\in \Bbb T^d}|f(\x)|$.

Let $\mathcal{T}_{n}^d$ be the  space of all multivariate
trigonometric polynomials $T_\a$  of  form
\begin{equation}\label{5.1}T_{\a}(\x)=\sum_{\|\kk\|_\infty\le n}a_\kk e^{i\kk\x}, \end{equation}where  $$\kk=(k_1,\dots,k_d)\in \Bbb Z^d, \ \
\|\kk\|_\infty=\max_{1\le i\le d}|k_i|\le n, \ \
\a=(a_{\kk})\in\Bbb R^N, \ \ N=\dim \mathcal T_n^d=(2n+1)^d.$$ The
classical Nikolskii inequality \cite[Theorem 3.4]{T} states that
there is a constant $C>0$ independent of $d$ such that for any
trigonometric polynomial $T_{\a}$ of degree $\le n$,
\begin{equation}\label{5.2}
\|T_{\a}\|_{q}\leq (C^dN)^{(1/p-1/q)_+}\|T_{\a}\|_{p},
\end{equation} for any $1\leq p, q\leq\infty$. The order $(1/p-1/q)_+$ in the
above inequality \eqref{5.2} is sharp. This means that
\begin{equation}\label{5.2-0}
N_{p,q}^{\rm wor}(\mathcal T_n^d)\asymp n^{d(1/p-1/q)_+},
\end{equation}i.e.,
the worst case  Nikolskii factors $N_{p,q}^{\rm wor}(\mathcal
T_n^d)$ is order  the $n^{d(1/p-1/q)_+}$.

Next we consider the average Nikolskii factors $N_{p,q}^{\rm
ave}(\mathcal T_N^d)$. In this section, we always suppose that
$$T_{\a}(\x)=\sum_{\|\kk\|_\infty\le n}a_\kk e^{i\kk\x}$$  is  a
random multivariate trigonometric polynomial, where $\a\sim
N(0,I_N)$. We remark that $$\{e^{i\kk\x}\}_{\|\kk\|_\infty\le n}$$
is an orthonormal basis for $\mathcal T_n^d$.

\begin{thm}\label{thm5.1}
Let $p=2, \,1\leq q\leq\infty$. We have
\begin{equation}\label{5.3}N^{\rm{ave}}_{2,q}(\mathcal{T}^{d}_{n})=
\mathbb{E}\Big(\frac{\|T_{\a}\|_{q}}{\|T_{\a}\|_{2}}\Big)\asymp
N^{-1/2}\Bbb E\|T_{\a}\|_q\asymp\left\{
\begin{aligned}
&1,&\ \,1\leq q<\infty,\\
&\sqrt{\ln N},&\ \,q=\infty,
\end{aligned}
\right.\end{equation} where $N=(2n+1)^{d}$, and the equivalent
constants are independent of $N$ and $d$.
\end{thm}
\begin{proof} We note that $\{e^{i\kk\x}\}_{\|\kk\|_\infty\leq n}$ is an
orthonormal basis for $\mathcal T^d_n$, and $$m(\x):=\sum_{\|\kk\|_\infty\le n}|e^{i\kk\x}|^{2}=(2n+1)^{d}=N.$$ For $1\leq s,q<\infty$,
\eqref{2.10} and (\ref{2.5}) give
$$\mathbb{E}\|T_{\a}\|^{s}_{q}\asymp \|m(\x)^{1/2}\|_q^s\asymp N^{s/2},$$
and $$ N^{\rm{ave}}_{2,q}(\mathcal{T}^{d}_{n})\asymp N^{-1/2}\Bbb
E\|T_{\a}\|_q\asymp 1.
$$

For $q=\infty$, by the Nikolskii inequality (\ref{5.2}), the
H\"{o}lder inequality,  and Theorem \ref{thm2.3}, we have for
$2<q<\infty$,
\begin{align*}
\mathbb{E}\|T_{\a}\|_{\infty}&\ll C^{d/q}N^{1/q}\mathbb{E}\|T_{\a}\|_{q}\leq C^{d/q}N^{1/q}(\mathbb{E}\|T_{\a}\|^{q}_{q})^{1/q} \\
&=C^{d/q}N^{1/q+1/2}C(q)\ll C^{d/q}N^{1/2+1/q}q^{1/2}.
\end{align*}
We remark that the  constants in the above inequalities are
independent of $N, d,$ and $q$. Thus, taking $q=\ln N=d\ln(2n+1)$,
we get
$$\mathbb{E}\|T_{\a}\|_{\infty}\ll (N\ln N)^{1/2}.$$ By
(\ref{2.5}) we obtain
\begin{equation}N^{\rm{ave}}_{2,\infty}(\mathcal{T}_{n})\ll\sqrt{\ln
N},\label{5.1-0}\end{equation}where the  constant in the above
inequality \eqref{5.1-0} is independent of $N$ and $d$.

Now we show the lower bound  of
 $N^{\rm{ave}}_{2,\infty}(\mathcal{T}^{d}_{n})$. We set
 $$X=(X_\kk),\ \ X_\kk=\frac1{\sqrt N}T_{\a}(\x_\kk),$$ and
 $$ \x_\kk=(x_{k_{1}},\ldots,x_{k_{d}}),\ x_{k_{i}}=\frac{2\pi k_i}{2n+1}, \ \ k_i=0, 1,\dots, 2n,\ 1\le i\le d,$$
 where $T_{\a}$ is the random trigonometric polynomial given by \eqref{5.1},
 $\a\sim N(0,I_N)$. Then $X$ is the Gaussian centered
 random vector with covariance matrix $C=(C_{\kk,\bf{l}})$,
 where $C_{\kk,\bf{l}}=\Bbb E(X_\kk \overline{X_{\bf{l}}})$. We have
\begin{align*}\Bbb E(X_\kk \overline{X_{\bf{l}}})&=\frac1N\Bbb E(T_{\a}(x_\kk)\overline{T_{\a}(x_{\bf{l}})})=\frac{1}{N}\sum_{\|\bf{s}\|_\infty\le n}e^{i\bf{s}(x_\kk-\x_{\bf{l}})}\\
&=\frac{1}{N}\prod^{d}_{i=1}\sum_{|s_{i}|\le n}e^{is_{i}
(x_{k_{i}}-x_{l_{i}})}\\
&=\frac
1N\prod^{d}_{i=1}D_n(x_{k_i}-x_{l_i})=\delta_{\kk,\bf{l}}.\end{align*}where
$D_n$ is the Dirichlet kernel,
$$D_n(x)=\sum_{|k|\le n}e^{ikx}=1+2\sum_{k=1}^n\cos kx= \frac{\sin(n+\frac12)x}{\sin \frac
x2}.$$ Hence, $X=(X_{\kk}):=(\widetilde{X_{i}})^{N}_{i=1}\sim N(0,I_N)$. This means that $\widetilde{X_1},\dots, \widetilde{X_N}$
are the i.i.d. random variables with $\widetilde{X_i}\sim N(0,1)$. It follows
from \cite{P, V} that
$$\Bbb E \|T_{\a}\|_\infty\ge \sqrt N \ \Bbb E\max_{1\le k\le
N}|\widetilde{X_k}|\asymp \sqrt N\, \sqrt{\ln N}.
$$
By \eqref{2.5} and \eqref{5.1-0} we obtain
$$N^{\rm{ave}}_{2,\infty}(\mathcal{T}^{d}_{n})\asymp \sqrt{\ln
N}.$$Theorem \ref{thm5.1} is proved. \end{proof}

\begin{thm} \label{thm5.2}
Let $r>0$. Then for $N> \max\{e^4, r+1\}$, we have
\begin{equation}\label{5.5}
\mathbb{E}\frac{1}{\|T_{\a}\|^{r}_{\infty}}\ll N^{-r/2}(\ln
N)^{-r/2},
\end{equation}
where $N=(2n+1)^{d}$, and
constant in the above inequality is independent of $N$ and $d$.
\end{thm}
\begin{proof}
As in the proof of Theorem \ref{thm5.1},
we set
$$X=(X_\kk),\ \ X_\kk=\frac1{\sqrt N}T_{\a}(\x_\kk),$$
and $$ \x_\kk=(x_{k_{1}},\ldots,x_{k_{d}}),\ x_{k_{i}}=\frac{2\pi k_i}{2n+1}, \ \ k_i=0, 1,\dots, 2n,\ 1\le i\le d.$$
It follows  that $X=(X_{\kk})=(\widetilde{X_{i}})^{N}_{i=1}\sim N(0,I_{N})$. Therefore,
$$\mathbb{E}\frac{1}{\|T_{\a}\|^{r}_{\infty}}
\leq N^{-\frac{r}{2}}\,\mathbb{E}\frac{1}{\max\limits_{1\leq
k\leq N} |\widetilde{X_{k}}|^{r}}.$$
From the proof of Theorem \ref{3.2}, we know
\begin{align*}
\mathbb{E}\frac{1}{\max\limits_{1\leq k\leq N}
|\widetilde{X_{k}}|^{r}}&\ll(\ln N)^{-\frac{r}{2}}.
\end{align*}
This deduces that
$$
\mathbb{E}\frac{1}{\|T_{\a}\|^{r}_{\infty}}\ll
N^{-\frac{r}{2}}(\ln N)^{-\frac{r}{2}},
$$
which completes the proof of Theorem \ref{thm5.2}.
\end{proof}

Finally, we get the estimates of the average
Nikolskii factors for random trigonometric polynomials on the $d$-dimensional
torus $\Bbb T^d=[0,2\pi]^d$.
\begin{thm}\label{thm5.3}
Let  $1\leq p,q\leq\infty$, $N=(2n+1)^{d}$. Then we have
\begin{equation}\label{5.6}
N^{\rm ave}_{p,q}(\mathcal{T}^{d}_{n})=\mathbb{E}\Big(\frac{\|T_{\a}\|_{q}}{\|T_{\a}\|_{p}}\Big)\asymp\left\{
\begin{aligned}
 &1,&\ \,1\leq p,q<\infty\ or\ p=q=\infty,\\
 &(\ln N)^{1/2},&\ \,1\leq p<\infty,q=\infty,\\
 &(\ln N)^{-1/2},&\ \,1\leq q<\infty,p=\infty,
\end{aligned}
\right.
\end{equation}
where the equivalent
constants are independent of $N$ and $d$.
\end{thm}
\begin{proof}
As in the proof of Theorem \ref{thm1.1} in Section 4, by the
H\"{o}lder inequality, Theorems \ref{thm2.2}, \ref{thm5.1} and
 \ref{thm5.2}, we obtain the upper bounds of
$N^{\rm{ave}}_{p,q}(\mathcal{T}^{d}_{n})$ for $1\le p,q\le \infty$
as follows.
\begin{equation}\label{5.7}
N^{\rm{ave}}_{p,q}(\mathcal{T}^{d}_{n})\ll \left\{
\begin{aligned}
 &1,&\ \,1\leq p,q<\infty\ or\ p=q=\infty,\\
 &(\ln N)^{1/2},&\ \,1\leq p<\infty,q=\infty,\\
 &(\ln N)^{-1/2},&\ \,1\leq q<\infty,p=\infty.
\end{aligned}
\right.
\end{equation}
By \eqref{4.7} and \eqref{5.7} we get
$$
N^{\rm{ave}}_{p,q}(\mathcal{T}^{d}_{n})\gg \left\{
\begin{aligned}
 &1,&\ \,1\leq p,q<\infty\ or\ p=q=\infty,\\
 &(\ln N)^{1/2},&\ \,1\leq p<\infty,q=\infty,\\
 &(\ln N)^{-1/2},&\ \,1\leq q<\infty,p=\infty,
\end{aligned}
\right.
$$
which gives the lower bounds of
$N^{\rm{ave}}_{p,q}(\mathcal{T}^{d}_{n})$ for $1\le p,q\le \infty$.

Theorem \ref{thm5.3} is proved.
\end{proof}
\begin{rem}
From \eqref{5.2-0} and \eqref{5.6} we know that the order of
the worst and average  Nikolskii factors for $1\leq q\leq
p<\infty$ or $p=q=\infty$ are same, which both are equivalent to
the constant $1$. Whereas, for other cases, the orders of average
Nikolskii factors are significantly smaller than the worst case
Nikolskii factors. More precisely, for $1\leq p<q<\infty$, the
worst case Nikolskii factor is order $N^{1/p-1/q}$,
while the average case Nikolskii factors is order 1;
for $1\leq p<q=\infty$, the worst case Nikolskii factor is order $N^{1/p}$, while the order of the average Nikolskii
factor is  $(\ln N)^{1/2}$; for $1\le q<p=\infty$, the worst case
Nikolskii factor is order 1, while the order of the
average Nikolskii factor is $(\ln N)^{-1/2}$.

This indicates that for $1\le p<q\le \infty$ or $1\le q<p=\infty$,
the average Nikolskii factor $N^{\rm ave}_{p,q}(\mathcal{T}^{d}_{n})$
 is significantly smaller than the worst Nikolskii factor
$N^{\rm wor}_{p,q}(\mathcal{T}^{d}_{n})$.
\end{rem}
\begin{rem}It is interesting to note that  the equivalence constants in \eqref{5.6} are independent of the dimension $d$.\end{rem}

\noindent\textbf{Acknowledgments}
  The authors  were
supported by the National Natural Science Foundation of China
(Project no. 12371098).


\begin{thebibliography}{99}

\bibitem{AAR}G.E. Andrews, R. Askey, R. Roy, Special Fuctions, Cambridge Univ. Press, Cambridge, 1999.

\bibitem{AD}V. V. Arestov, M. V. Deikalova, Nikolskii inequality for algebraic polynomials on a multidimensional Euclidean sphere,
 Proc. Steklov Inst. Math. 284 (Suppl. 1) (2014) 9-23.

\bibitem{Bos}L. Bos, Markov factors on average-An $L_{2}$ case, J. Approx. Theory 241 (2019) 1-10.

\bibitem{BL}P. Borwein, R. Lockhart, The expected $L_{p}$ norm of random polynomials, Proc. Amer. Math. Soc. 129 (2001) 1463-1472.

\bibitem{BLR}S. Byun, J. Lee, T. R. Reddy,  Zeros of
random polynomials and their higher derivatives. Trans. Amer.
Math. Soc. 375 (9) (2022) 6311-6335.

\bibitem{B}V.I. Bogachev, Gaussian Measures, Mathematical surveys and Monographs, 62. Providence: Amer Math Soc, 1998.

\bibitem{CP} L. Coutin, L. Peralta,  Rates of convergence for the
number of zeros of random trigonometric polynomials, Bernoulli 29
(3) (2023) 1983-2007.

\bibitem{CE}S. Choi, T. Erd\'elyi, Average Mahler's measure and
$L_p$ norms of Littlewood polynomials. Proc. Amer. Math. Soc. Ser.
B 1 (2014) 105-120.

\bibitem{DM}A. Defant, M. Mastylo,  Norm estimates for random
polynomials on the scale of Orlicz spaces,  Banach J. Math. Anal.
11 (2) (2017)  335-347.

\bibitem{DX}C.F. Dunkl, Y. Xu, Orthogonal Polynomials of Several Variables,
Cambridge Univ. Press, 2001.

\bibitem{DW}F. Dai, H. Wang, Linear $n$-widths of diagonal matrices in the average and probabbilistic settings, Journal of Functional Analysis. 262 (2012), 4103-4119.

\bibitem{D}M. V. Deikalova, On the sharp Jackson-Nikolskii inequality for algebraic polynomials on a multidimensional Euclidean sphere,
Proc. Steklov Inst. Math. 266 (1) (2009) 129-142.



 \bibitem{DNN} Y. Do, H. Nguyen, O. Nguyen, Random trigonometric polynomials: universality and non-universality
 of the variance for the number of real roots, Ann. Inst. Henri Poincar\'e Probab. Stat. 58 (3) (2022) 1460-1504.

\bibitem{DNV} Y. Do, O. Nguyen, V. Vu,  Random orthonormal polynomials:
local universality and expected number of real roots, Trans. Amer.
Math. Soc. 376 (9) (2023), 6215-6243.

\bibitem{DT}Z. Ditzian, S. Tikhonov, Ulyanov and Nikolskii-type inequalities, J. Approx. Theory 133 (1) (2005) 100-133.

\bibitem{F}G.T. Fielding, The expected value of the integral around the unit circle of a certain class of polynomials, Bull. London Math. Soc. 2 (1970) 301-306.

\bibitem{GT}M. Ganzburg, S. Tikhonov, On sharp constants in Bernstein-Nikolskii inequalities, Constr. Approx. 45 (3) (2017)  449-466.

\bibitem{G}P.G. Grigoriev, Estimates for norms of random polynomials and their application, Math. Notes  69 (6) (2001).

\bibitem{LT} M. Ledoux,  M. Talagrand, Probability in Banach
spaces,  Classics in Mathematics, Springer-Verlag, Berlin, 2011.

\bibitem{LWW}J. Li, H. Wang, K. Wang, Weighted $L_{p}$ Markov factors with doubling weights on the ball, J. Approx. Theory 294 (2023) 105939.

\bibitem{M}V.E. Maiorov, Kolmogorov's $(n,\delta)$-widths of spaces of the smooth functions, Russian Acad. Sci. Sb. Math. 79 (2) (1994) 265-279.

\bibitem{NB}D. J. Newman, J. S. Byrnes, The $L_4$ norm of a polynomial with coefficients $\pm1$, MAA Monthly 97 (1990) 42-45.

\bibitem{P}G. Pisier, The Volume of Convex Bodies and Banach Space
Geometry, Cambridge Univ. Press, Cambridge, UK, 1989.

\bibitem{Pr}I. E. Pritsker, Asymptotic zero distribution of random
polynomials spanned by general bases, Contemp. Math. 661 (2016)
121-140.

\bibitem{PR} I.  Pritskera, K. Ramachandran, Random Bernstein-Markov
factors, J. Math. Anal. Appl. 473 (2019) 468-478.

\bibitem{SZ}R. Salem, A. Zygmund, Some properties of trigonometric series whose terms have random signs, Acta Math. 91(1954) 245-301.

\bibitem{SW}S. Steinerberger, H. Wu,  On zeroes of random polynomials
and an application to unwinding. Int. Math. Res. Not. 13 (2021)
10100-10117.

\bibitem{T}V. Temlyakov, Multivariate Approximation, Cambridge university press, 2018.

\bibitem{V}R. Vershynin, High-dimensional probability: An introduction with applications in data science, volume 47. Cambridge university press, 2018.

\bibitem{WYZ}H. Wang, W. Ye, X. Zhai, Average case weighted $L_{2}$ Markov factors with doubling weights, J. Approx.
Theory 256 (2020) 1-8.

\bibitem{WZ}H. Wang, X. Zhai, Best approximation of functions on the ball on the weighted Sobolev space equipped with a Gaussian measure, J. Appro. Theory 162 (2010), 1160-1177.

\end{thebibliography}
\end{document}